\documentclass[a4paper,11pt,reqno]{amsart}
\usepackage{latexsym}
\usepackage{amssymb,amsfonts,amsmath,mathrsfs}
\usepackage{epstopdf}
 \usepackage{hyperref, amscd, flafter,epsf}
 \usepackage{graphicx}
 \usepackage{amssymb}
 \usepackage{amsbsy}
 \usepackage{amsmath}

\newtheorem{theorem}{Theorem}
\newtheorem{lemma}{Lemma}

 \numberwithin{equation}{section}

\def\bysame{\leavevmode\hbox to3em{\hrulefill}\thinspace}
\newcommand{\be}{\begin{equation}}
\newcommand{\ee}{\end{equation}}
\newcommand{\s}{\sigma}

\renewcommand{\b}{\beta}
\renewcommand{\r}{\rho}
\newcommand{\g}{\gamma}
\newcommand{\z}{\zeta}
\newcommand{\tOmega}{\widetilde{\Omega}}
\renewcommand{\a}{\alpha}
\renewcommand{\d}{\delta}
\newcommand{\e}{\epsilon}
\renewcommand{\l}{\log}
\renewcommand{\L}{\Lambda}

\renewcommand{\Re}{{\mathfrak{R}\,}}
\renewcommand{\Im}{{\mathfrak{I}\,}}
\newcommand\ds{\displaystyle}
\newcommand\zpz{\frac{\zeta'}{\zeta}}

\def\AB{\bigg|}
\def\({\left(}
\def\){\right)}

\def\intl{\int\limits}
\def\bOmega{{\boldsymbol \Omega}}
\newcommand{\meas}{\mathop{\rm meas}}

\renewcommand\be{\begin{equation}}
\renewcommand\ee{\end{equation}}
\newcommand\bea{\begin{eqnarray}}
\newcommand\eea{\end{eqnarray}}
\newcommand\bi{\begin{itemize}}
\newcommand\ei{\end{itemize}}
\newcommand\ben{\begin{enumerate}}
\newcommand\een{\end{enumerate}}
\newcommand\bes{\begin{equation*}}
\newcommand\ees{\end{equation*}}
 \def\C {{\mathbb C}}

\def\R {{\mathbb R}}
\def\E {{\mathcal E}}

\title[The distribution of $\zeta'/\zeta(s)$]
 {The distribution of the logarithmic derivative \\
 of the Riemann zeta-function} 
\author{S. J. Lester}

\address{Department of Mathematics, University of Rochester, Rochester, NY 14627 USA}

\curraddr{School of Mathematical Sciences, Tel Aviv University, Tel Aviv, 69978 Israel}

\email{slester@post.tau.ac.il}


\thanks{The author was supported in part by the NSF grant DMS-1200582.}

\begin{document}
\maketitle

\begin{abstract}
We investigate the distribution of the logarithmic derivative of
the Riemann zeta-function on the line $\Re(s)=\s$, where $\sigma$
lies in a certain range near the critical line $\s=1/2$.
For such $\sigma$, we show that the distribution of
$\zeta'/\zeta(s)$ converges to a two-dimensional
Gaussian distribution in the complex plane.  
Upper bounds on the rate of convergence to
the Gaussian distribution are also obtained.
\end{abstract}

\section{Introduction}
Let $\z(s)$ denote the Riemann zeta-function with $s=\s+it$ a complex variable. 
Throughout we let $T$ denote a sufficiently large parameter.

In unpublished work, A. Selberg proved that the logarithm of the Riemann 
zeta-function on and near the line $\s=1/2$ is normally distributed.  For 
$0\leq (2\s-1) =o(1)$ as $T \rightarrow \infty$, $\Psi(\s)=\frac12 \sum_{p \leq T} p^{-2\s}$, and real numbers $a<b$, he showed that
\[
 \meas\bigg\{ t \in [0, T] \, : \, \l |\z(\s+it)| \Psi(\s)^{-1/2} \in [a,b] \bigg\}
 =\frac{T}{2\pi} \int_a^b \! e^{-x^2/2} \, dx +O\bigg( T \frac{(\l \Psi(\s))^2}{\sqrt{ \Psi(\s)}}\bigg)
\]
and
\[
 \meas\bigg\{ t \in [0, T] \, : \, \arg \z(\s+it) \Psi(\s)^{-1/2} \in [a,b] \bigg\}
 =\frac{T}{2\pi} \int_a^b \! e^{-x^2/2} \, dx +O\bigg( T \frac{\l \Psi(\s)}{\sqrt{\Psi(\s)}}\bigg),
\]
where $\meas$ denotes Lebesgue measure (see ~\cite{S}).
Although Selberg did not publish proofs of these results, his student K. M. Tsang gave the details of Selberg's argument 
in his PhD thesis ~\cite{Ts}.  These theorems may also be proved, albeit with larger error terms, by  the method of A. Ghosh in ~\cite{Gh1}
and ~\cite{Gh}.

The purpose of this article is to investigate the distribution of the logarithmic
derivative of the Riemann zeta-function near the critical line $\s=1/2$. 

The distribution of $\zeta'/\zeta(\sigma+it)$ was also  studied by C. R. Guo~\cite{G},
who showed the following.  Write $Q(x,y)=\tilde Q(x+iy)$, where  $\tilde Q: \mathbb{C} \rightarrow \mathbb{R}$ is infinitely differentiable in $x$ and $y$ and has compact support. Then for any
$0< \e <1/6$ and  $\s \in [1/2+(\l T)^{-(1/6-\e)}, 2]$ 
Guo showed that
\[
\frac1T \int_0^T \! \tilde Q\bigg(\frac{\z'}{\z}(\s+it)\bigg) \, dt= 
\iint_{\mathbb{R}^2} \! Q(x,y)f(x,y) \, dx dy+E(T, Q,\s, \e).
\] 
Here $f(x,y)$ is the Fourier transform of 
\[
\ds \prod_p 
\bigg( \int_0^1 \exp\bigg(2\pi i u \l p \sum_m \frac{\cos(2 \pi m t)}{p^{m\s}}-2\pi i v \l p 
\sum_m \frac{\sin(2 \pi m t)}{p^{m\s}} \bigg)\, dt \bigg)
\]
and 
\[
E(T, Q,\s, \e) \ll \exp(-\tfrac14(\l T)^{\frac23(\frac16-\e)}) \iint_{A} 
\!|\widehat{Q}(\a, \b)  |\, d\a d\b+\iint_{\mathbb{R}^2 \setminus A} \!|\widehat{Q}(\a, \b)  |\, d\a d\b,
\]
where $A$ is the square
\[
A= [-(\l T)^{3\e-\eta},\ (\l T)^{3\e-\eta}   ] \times  [-(\l T)^{3\e-\eta},\ (\l T)^{3\e-\eta}   ]
\]
with $\eta>0$ arbitrary.
 
In this paper we calculate the distribution function of $\z'/\z(\s+it)$
for $\s=1/2+\psi(T)/\l T$ with   $\psi(T)$  any positive function tending to infinity  
in such a way that  $\psi(T)= o(\l T)$ as $T\to\infty$. We also show that
in this range the distribution of $\z'/\z(\s+it)$ converges
to a two-dimensional Gaussian distribution in the complex plane.   
Finally, for rectangles
 with sides parallel to the coordinate axes
and for disks centered at the origin, we give explicit
upper bounds on the rate of convergence to the normal
distribution.  We should mention that our results are consistent with Guo's work:
one can show that the probability density function $f(x,y)$ 
in Guo's theorem does indeed converge to a two-dimensional
Gaussian when $2\s-1=o(1)$ as $T \rightarrow \infty$.

Throughout we write $\vec{u}=(u,v) \in \mathbb{R}^2$ and, if  $z \in \mathbb{C}$, 
we use the non-standard notation $\vec{u} \cdot z =u \Re(z)+v \Im(z)$.   We define $\mathbf{1}_A(\a)$ 
to be the indicator function of the set $A$, which is equal to one if $\a \in A$ and 
is equal to zero if $\a \notin A$.  
If $\theta=(\theta_1, ..., \theta_n) \in [0,1]^n$, we write
\[
\int\limits_{\mathbb{T}^n} \! F(\theta) \, d\theta=
\int_0^1 \cdots \int_0^1 F(\theta_1, ..., \theta_n) \, d\theta_1 \cdots d\theta_n.
\]
Finally, for $f \in L^1(\R^n)$ we define the Fourier transform of $f$ by
\[
\widehat f(x)=\intl_{\R^n} f(\xi) e(-x \cdot \xi) \, d\xi,
\]
where $e(x)=e^{2\pi ix}$.

\section{Main Results and a Summary of the Method}

Let $\psi(T)=(2\s-1) \l T$ and for $\psi(T) \geq 1$ define
\[
V=V(\s)= \ds \frac12 \sum_{n=2}^{\infty} \frac{\L^2 (n)}{n^{2\s}}. 
\]
When  $\psi(T)$  tends to infinity with $T$, we may think of $V$ 
as the variance of $ \z'/\z(\s+it)$. For if $\s$ is in this range and the Riemann 
hypothesis is true, it follows from equation (1.2) of Selberg~\cite{Sel} that  
\[
\frac1T \int_0^T \! \bigg| \frac{\z'}{\z}(\s+it) \bigg|^2 \, dt \sim \sum_{n=2}^{\infty} \frac{\L^2 (n)}{n^{2\s}}.
\]
With this in mind, we consider the normalized function $\z'/\z(\s+it)V^{-1/2}$.  Our first main result is
\begin{theorem} \label{theorem 1.1}
Let $\psi(T)=(2\s-1) \l T$, and 
\[
\bOmega=e^{-10}\min\big(V^{3/2},(\psi(T)/ \l \psi(T))^{1/2}\big).  
\]
Suppose that $\psi(T)\to \infty$ with $T$,  $\psi(T)=o(\l T)$,
 and that $R$ is  a rectangle in $\mathbb C$ whose sides are parallel to the coordinate axes
and have length greater than $\bOmega^{-1}$. 
Then we have
\be \notag
\begin{split}
\meas\bigg\{t \in (0, T) : \frac{\z'}{\z}(\s+it) V^{-1/2} \in R \bigg\}
=\frac{T}{2\pi}\iint_R \! e^{-(x^2+y^2)/2} \, dx \, dy +O\left(T\frac{(\meas(R)+1)}{\bOmega}\right).
\end{split}
\ee
\end{theorem}
In the range 
	\[
	\bigg(\frac{\l \l T}{ \l T }\bigg)^{1/7}\ll (2\s-1) =o(1),
	\]
 the error term is of order 
$T (\meas(R)+1)V^{-3/2}$, while for $\s$ closer to $1/2$, it is of size 
$T(\meas(R)+1) (\l \psi(T)/\psi(T))^{1/2}$.  At the cost of a longer proof, the
condition that the length of each side of the rectangle $R$ should be  greater than $\bOmega^{-1}$ 
could be removed.

It does not seem possible to prove a distribution theorem for  $\z'/\z(\s+it)$ when $ (2\s-1) \l T \ll 1$ 
without the assumption of some unproven hypothesis. For in this range the    
moments of  $\zeta'/\zeta(\s+it)$ depend on  correlations
 of the zeros of the Riemann zeta-function.  
Goldston, Gonek, and Montgomery ~\cite{Go} 
proved, under the assumption of the Riemann hypothesis, that for $T^{-1} \l^3 T \leq a \ll 1$,  
\[
\int_0^T \, \bigg|  \frac{\z'}{\z}\Big(\frac12+\frac{a}{\l T}+it\Big)\bigg|^2 \, dt 
\sim \bigg(\frac{1-e^{-2a}}{4a^2} + \int_1^{\infty} (F(\a, T)-1)e^{-2a\a} d\a \bigg) T \l^2 T.
\]
Here $F(\a, T)$ is defined by
\[
F(\a, T)=\frac{1}{\frac{T}{2\pi}\l T} \sum_{0 < \g, \g' \leq T} T^{i\a(\g-\g')}w(\g-\g'),  
\]
where $w(x)=4/(4+x^2)$ and the sum is over pairs of ordinates of zeros
of the Riemann zeta-function. (For more on $F(\a, T)$ see ~\cite{M}.) 
Moreover, D. W. Farmer \emph{et al.} ~\cite{F} have recently proved that if the 
Riemann hypothesis  and some additional plausible hypotheses about the zeros are true, 
then for $(2\s-1) \l T \approx 1$,  
 the even moments of $|\z'/\z(\s+it)|$ may be expressed in terms of   correlations of the zeros
of $\z(s)$.  The form of the answers suggests that  $\z'/\z(\s+it)V^{-1/2}$ is unlikely to be normally
distributed when $(2\s-1)\l T \approx 1$.

In applications it is useful to have an analogue of 
Theorem \ref{theorem 1.1} in a disk.
For this reason, we also prove
\begin{theorem} \label{theorem disk 1}
Let $\psi(T)=(2\s-1) \l T$, 
\[
\bOmega=e^{-10} \min\big(V^{3/2}, (\psi(T)/ \l \psi(T))^{1/2}\big).  
\]
Suppose that $\psi(T)\to \infty$ with $T$,  $\psi(T)=o(\l T)$,
 and that  $r$ is a real number such that $r \bOmega \geq 1$.
Then we have
\be \label{main measure 1}
\begin{split}
\meas \bigg\{ t \in (0, T) : \bigg|\frac{\z'}{\z}(\s+it)  \bigg| \leq \sqrt V r  \bigg \}
=  \, T(1-e^{-r^2/2}) +O\(T\(\frac{r^2+r}{\bOmega}\)\).
\end{split}
\ee
If, in addition, we let $\widetilde \bOmega= 
\min\big((2\s-1) e^{\s/(2\s-1)},e^{-10}(\psi(T)/ \l \psi(T))^{1/2}\big)$,   then we have  for $r \widetilde \bOmega \geq 1$
\be \label{small measure 1}
\meas\bigg\{ t \in [0, T] : \bigg|\frac{\z'}{\z}(\s+it) \bigg| \leq \sqrt{V} r \bigg\} 
\ll Tr^2.
\ee
\end{theorem}

One of the two main components of the proofs
of Theorem \ref{theorem 1.1} and Theorem \ref{theorem disk 1} is
an approximate formula for the characteristic function (ch.f.) of  
$ \z'/\z(\s+it)V^{-1/2}$.
\begin{theorem} \label{theorem 2}
Let $\psi(T)=(2\s-1) \l T$, and 
\[
\Omega= e^{-10}\min\big(V^{1/2},(\psi(T)/ \l \psi(T))^{1/2}\big).
\]
Suppose that $\psi(T)\to \infty$ with $T$,  $\psi(T)=o(\l T)$, and  $ |u|,|v|\leq \Omega$. 
Then 
\bes \label{c.f. formula for ld}
\frac1T \int_0^T \, e\bigg(\vec{u} \cdot\frac{\z'}{\z}(\s+it)V^{-1/2}\bigg) \, dt
=e^{-2 \pi^2(u^2+v^2)}(1+\mathcal E_A(u,v))+\mathcal E_B,
\ees
where 
\[
\mathcal E_A(u,v)\ll \frac{(|u|+|v|)^3}{V^{3/2}} +\frac{u^2+v^2}{\psi(T)^{10}} \qquad
 \mbox{and} \qquad \mathcal E_B \ll  \psi(T)^{-10}.
\]
\end{theorem}
Observe that in Theorem \ref{theorem 2} we 
may take $u$ or $v$ to be zero and obtain    
approximate formulas for the ch.f.   of 
$\Re \z'/\z(\s+it)$ and $\Im \z'/\z(\s+it)$.
Using these, we could easily prove that  analogues of Theorem~\ref{theorem 1.1}
hold for both the real and  imaginary  parts of $\z'/\z(s)$.

	Theorem \ref{theorem 2} implies that if $(2\s-1) \l T$   tends to 
infinity with $T$ and is also $o(\l T)$, then the ch.f. of 
$\z'/\z(\s+it)V^{-1/2}$ converges pointwise, as $T \rightarrow \infty$, 
to a two-dimensional Gaussian. Hence, for $\s$ in 
this range, it immediately follows
from standard probability theory 
that for any Borel measurable 
$\mathcal S \subset \C$ with positive Jordan content,
\[
\lim_{T \rightarrow \infty} \frac1T
\meas\bigg\{t \in [0, T] : \frac{\z'}{\z}(\s+it) V^{-1/2} \in \mathcal S \bigg\}
=\frac{1}{2\pi}\iint_{\mathcal S} \! e^{-(x^2+y^2)/2} \, dx \, dy.
\]  
To obtain Theorem
\ref{theorem 1.1} from
Theorem \ref{theorem 2}, we use Beurling-Selberg
functions, which are analytic approximations of
the signum function. They are also 
integrable along the real axis and have
Fourier transforms that vanish outside of an interval.
In the 1930's, Beurling discovered these functions but never
published his findings.  Independently, Selberg rediscovered
them and   used them in several contexts, one of which
was the study of the distribution of    $\log \z(s)$. 
For a discussion of these functions see Selberg ~\cite{Selb}.

The proof of Theorem \ref{theorem disk 1} is similar.
However, in this case we use Beurling-Selberg
functions that approximate the
indicator function of a disk.
These functions were introduced by Holt and Vaaler in ~\cite{Ho}, and 
their existence is
 a special case of a general theorem on Beurling-Selberg functions
 for balls in Euclidean space.  

To prove Theorem \ref{theorem 2} we start with a formula for $ \z'/\z(s)$ 
that was proved by Selberg in ~\cite{Se}. 
For $\s+it$ not too close to a zero of the
Riemann zeta-function, this formula expresses $\z'/\z(\s+it)$ 
as, essentially, the 
Dirichlet polynomial $ -\sum_{n \leq x}  {\L(n)}{n^{-\s-it}}$.  
If $\psi(T)=(2\s-1)\l T$ tends to infinity with $T$,  this formula holds for most $t$.  
This allows us to reduce our problem to calculating the 
ch.f. of $ -\sum_{n \leq x}  {\L(n)}{n^{-s}}$,  and we accomplish that 
by computing its  moments.

\section{The Characteristic Function of $\z'/\z(\s+it)$}
Our initial step is to express the ch.f. of $ \z'/\z(\s+it) V^{-1/2}$ 
in terms of the ch.f. of the Dirichlet
polynomial $ -V^{-1/2}\ \sum_{n \leq x}  {\L(n)}{n^{-\s-it}}$. 
Near a zero of the Riemann zeta-function $\z'/\z(s)$ cannot be approximated
by a Dirichlet polynomial, hence we need to bound
the possible contribution of $t \in (0, T)$, where $\s+it$
is close to a zero of $\z(s)$. In ~\cite{Se} Selberg
discovered a very clever way to do this.
Selberg's
 approach begins with an explicit formula for 
$\z'/\z(s)$. The explicit formula
contains two important terms, the first of which is a 
sum over primes similar to 
the Dirichlet polynomial above, while the second
is a sum over zeros of the Riemann zeta-function. 
Selberg showed  
that for certain values of $t$, the contribution 
of the sum over zeros
can be bounded in terms of a sum over primes.
Using this formula, we will prove 
\begin{lemma} \label{c.f. 1}
Suppose that  $10 \leq x \leq T^{1/18}$  and $1/2+4/\l x \leq \s \leq 2$. Then
\bes
\frac1T \int_0^T \! e\bigg(-\vec{u} \cdot \frac{\z'}{\z}(\s+it) V^{-1/2} \bigg) \, dt= 
\frac1T \int_0^T \, e\bigg(\vec{u} \cdot \sum_{n \leq x} \frac{\L(n)}{n^{\s+it}}V^{-1/2} \bigg) \, dt+ E_1
\ees
where 
\[
E_1 \ll (|u|+|v|)V^{-1/2}x^{(1/2-\s)/2} \l T+T^{-(\s-1/2)/3} \frac{\l T}{\l x}.
\]
\end{lemma}
Throughout we let
$\r=\b+i\g$ denote 
a zero of the Riemann zeta-function.
To state
the explicit formula for $\z'/\z(s)$,
we first define  the number 
\be \label{sxt}
\s_{x, t}= \frac12+2 \max\Big(\b-\frac12, \frac{2}{\l x} \Big),
\ee
where $x \geq 2$ and $t>0$.
Here the maximum is taken over zeros $\r$ satisfying 
$\ds |t-\g| \leq x^{3|\b-1/2|}/\l x$.  
For $\s \geq \s_{x,t}$ and $2 \leq x \leq t^2$, A. Selberg proved 
(see equation  4.9 in ~\cite{Se}) that
\be \label{selberg}
\begin{split}
-\frac{\z'}{\z}(\s+it)=\sum_{n \leq x^3} \frac{\L(n)}{n^{\s+it}}w_x(n)+&
O\bigg(x^{ (1/2-\s)/2}\bigg|\sum_{n \leq x^3} \frac{\L(n)}{n^{\s_{x,t}+it}}w_x(n) \bigg| \bigg)\\
+&O(x^{ (1/2-\s)/2} \l t),
\end{split}
\ee
where 
\bes
w_x(n)=
\begin{cases} 
  1 & \mbox{if $n \leq x,$} \\
  \frac{\l^2 ({x^3}/{n})-2\l^2 ({x^2}/{n})}{2 \l^2 x} & \mbox{if $x < n \leq x^2,$} \\
  \l^2 ({x^3}/{n}) & \mbox{if $x^2 < n \leq x^3,$} \\
  0 & \mbox{if $n>x^3$}.
\end{cases}
\ees
One can easily modify Selberg's proof to show that an analogue of \eqref{selberg} holds
where $\s_{x,t}$ is replaced by $\s$ in the error term and we shall use this later in the proof of
Lemma \ref{c.f. 1}. (This fact merely simplifies
the proof.)

We will now show that if $(2\s-1) \l T$   tends to 
infinity with $T$ and is also $o(\l T)$, then  the measure of the 
set of $t \in (0, T)$ for which $\s_{x,t}>\s$ is 
$o(T)$.  Hence, for $\s$ in this range, \eqref{selberg}
holds almost everywhere, in the sense that
the proportion of $t \in (0, T)$ for which
the formula does not hold tends to zero as $T \rightarrow \infty$. 
We prove this by using a zero-density estimate of Jutila~\cite{Ju}, that is, 
an estimate for the number of zeros of $\z(s)$  with  
$\b>\s$ and $0< \g < T$.  Jutila's result is that for any $\e>0$,
\be\label{zero density 11}
N(\s, T)=\sum_{\substack{ 0<\g < T \\ \b >\s }} 1 \ll T^{1-(1-\e)(\s-1/2)} \l T.
\ee
We are now ready to prove   
\begin{lemma} \label{key lem}
Let $1/2+4/\l x \leq \s \leq 2$  and, for any fixed $0<\e<1$, 
let $ 10 \leq x \leq T^{\e/3}$. Then 
\bes
\meas\{ t \in [2, T] : \s_{x,t}> \s \} \ll 
T^{1-(1/2-\e)(\s-1/2)}\frac{\l T}{\l x},
\ees
where $\s_{x,t}$ is the number defined in \eqref{sxt}, 
and the implied constant depends only on $\e$.
\end{lemma}
\begin{proof}
By the definition of $\s_{x,t}$,  if for some $t\geq 2$ we have that $\s_{x,t} > \s$,  
then there is a zero $\r_0$ such that $\b_0 >(\s-1/2)/2+1/2$ 
and $\ds |t-\g_0| \leq x^{3|\b_0-1/2|}/\l x$.  Furthermore, 
for each such $\r_0$ we have
\bes
\meas\bigg\{t \in [2, T] : |t-\g_0| \leq \frac{x^{3|\b_0-1/2|}}{\l x} \bigg\} 
\leq 2 \frac{x^{3|\b_0-1/2|}}{\l x}.
\ees
We also observe that if $2 \leq t \leq T$ and $x \leq T^{\e/3} < T^{1/3}$, 
then  $ - T^{1/2}/ \l T \leq \g_0 \leq T+  T^{1/2}/ \l T$. 
Now let $\s'=(\s-1/2)/2+1/2$. Combining our  observations, we find that
\be\label{zero bd 11}
\meas\{t \in[2,T] : \s_{x, t} > \s \} \ll 
\sum_{\substack{  -\frac{ T^{1/2}}{ \l T} \leq \g \leq T+\frac{ T^{1/2}}{ \l T} \\ \b > \s'}} 
\frac{x^{3(\b-1/2)}}{\l x}.
\ee
Using \eqref{zero density 11}, we see that
\be \notag 
\begin{split}
\sum_{\substack{  0 < \g \leq 2T \\ \b > \s'}} x^{3(\b-1/2)}  
 =&-(N(v, 2T) x^{3(v-1/2)})\biggr|_{\s'}^1+3\l x \int\limits_{\s'}^1 \! x^{3(v-1/2)}  N(v, 2T) \, dv \\  
  \ll & (T^{1-(1-\e)(\s'-1/2)} \l T) (T^{\e(\s'-1/2)})\\
&\qquad+ x^{-3/2} T^{1+ (1-\e)/2} (\l x) (\l T )
\int\limits_{\s'}^1 \!     (x^3T^{\e-1})^{v} \, dv \\  
  \ll& T^{1-(1-2\e)(\s'-1/2)} \l T.
\end{split}
\ee 

The zeros of $\z(s)$ are symmetric about the real axis, so by our previous estimate, 
\bes
 \sum_{\substack{  -\frac{ T^{1/2}}{ \l T} 
 \leq \g \leq T+\frac{ T^{1/2}}{ \l T} \\ \b > \s'}} \frac{x^{3(\b-1/2)}}{\l x} \leq \frac{2}{\l x} 
 \sum_{\substack{  0 < \g \leq 2T \\ \b > \s'}} x^{3(\b-1/2)} 
 \ll T^{1-(1/2-\e)(\s-1/2) } \frac{\l T}{\l x}.
\ees
The result follows from this and \eqref{zero bd 11}.
\end{proof}

We now are ready to prove Lemma \ref{c.f. 1}.

\begin{proof}[Proof of Lemma \ref{c.f. 1}]
Let $B=\{t \in [0, T] : \s \geq \s_{x, t}\}$. 
By our comment after \eqref{selberg}, for $t \in B$ 
we may replace $\s_{x,t}$ by $\s$ in the first error term in \eqref{selberg}. 
Also note that $|e(\a)-e(\b)| =2 \pi |  \int_{\a}^{\b} \! e(x) \, dx| \leq 2 \pi |\b-\a|$. 
Hence, by these observations and \eqref{selberg}, we see that
\bes
\int\limits_{B} \! e\bigg(-\vec{u} \cdot \frac{\z'}{\z}(\s+it) V^{-1/2}\bigg) \, dt
=\int\limits_{B} \! e\bigg(\vec{u} \cdot \sum_{n \leq x^3} \frac{\L(n)}{n^{\s+it}} w_x(n) V^{-1/2}\bigg) \, dt+E_2,
\ees
where
\bes
E_2 \ll (|u|+|v|)V^{-1/2} x^{ (1/2-\s)/2}
\bigg(\int_0^T \! \bigg |\sum_{n \leq x^3} \frac{\L(n)}{n^{\s+it}}w_x(n)\bigg|  \, dt+T \l T\bigg).
\ees
Recall that $w_x(n)=1$ for $n \leq x$ and that $0 \leq w_x(n) \leq 1$ always. 
Estimating the integral using Cauchy's inequality and then applying 
Montgomery and Vaughan's mean value theorem for Dirichlet polynomials 
~\cite{MVMVT}, we see that  
\bes
\begin{split}
E_2 \ll &(|u|+|v|) V^{-1/2} x^{ (1/2-\s)/2}
\bigg(T \bigg( \sum_{n \leq x^3}\frac{\L^2(n)}{n^{2\s}}\bigg)^{1/2}+ T\l T\bigg) \\
\ll & T (|u|+|v|) V^{-1/2} x^{ (1/2-\s)/2} \l T,
\end{split}
\ees
where the estimate of the sum follows 
from a calculation using the Prime Number Theorem.  Noting that $x\leq T^{1/18}$,
we have
\bes 
\begin{split}
\int\limits_{B} \! e\bigg(-\vec{u} \cdot \frac{\z'}{\z}(\s+it) V^{-1/2}\bigg) \, dt
=&\int\limits_{B} \! e\bigg(\vec{u} \cdot \sum_{n \leq x^3} \frac{\L(n)}{n^{\s+it}} w_x(n) V^{-1/2}\bigg) \, dt \\
&+O\big(T(|u|+|v|)V^{-1/2}x^{ (1/2-\s)/2} \l T \big).
\end{split}
\ees
Next, note that $|e(x)| = 1$ and, by Lemma \ref{key lem}, that
\[
\meas([0,T]\setminus B) \ll T^{1-  (\s-1/2)/3} \frac{ \l T}{\l x}.
\]
Thus, 
\bes
\begin{split}
\int_0^T \! e\bigg(-\vec{u} \cdot \frac{\z'}{\z}(\s+it) V^{-1/2}\bigg) \, dt=
\int\limits_{B} \! e\bigg(-\vec{u} \cdot \frac{\z'}{\z}(\s+it) V^{-1/2}\bigg) \, dt
+O\(T^{1- (\s-1/2)/3} \frac{ \l T}{\l x}\).
\end{split}
\ees
An analogue of this formula also obviously holds for 
$ V^{-1/2}\sum_{n \leq x}   {\L(n)}n^{-(\s+it)}$. Hence,
we obtain
\be \label{summary}
\begin{split}
\int_0^T \! e\bigg(-\vec{u} \cdot \frac{\z'}{\z}(\s+it) V^{-1/2}\bigg) \, dt
=& \int_0^T \! e\bigg(\vec{u} \cdot \sum_{n \leq x^3} \frac{\L(n)}{n^{\s+it}} w_x(n) V^{-1/2}\bigg) \, dt \\
&+O\big(T(|u|+|v|)V^{-1/2}x^{ (1/2-\s)/2} \l T \big)\\
&+O\(T^{1- (\s-1/2)/3} \frac{ \l T}{\l x}\).
\end{split}
\ee
  	 
We would next like to replace the weight $w_x(n)$ by $1$. Write
\bes
\int_0^T \! e\bigg(\vec{u} \cdot \sum_{n \leq x^3} \frac{\L(n)}{n^{\s+it}} w_x(n) V^{-1/2}\bigg) \, dt
=\int_0^T \! e\bigg(\vec{u} \cdot \sum_{n \leq x} \frac{\L (n)}{n^{\s+it}}V^{-1/2}\bigg) \, dt+E_3.
\ees
Using the estimate $|e(\a)-e(\b)| \ll |\b-\a|$, we see that
\bes
E_3 \ll (|u|+|v|) V^{-1/2} \int_0^T \,\bigg|\sum_{n \leq x^3} \frac{\L(n)}{n^{\s+it}} w_x(n) -   \sum_{n \leq x} \frac{\L (n)}{n^{\s+it}}\bigg| \, dt.
\ees
By Cauchy's inequality and Montgomery and Vaughan's mean value theorem for Dirichlet polynomials ~\cite{MVMVT}, we find that
\bea \notag
E_3 &\ll& T (|u|+|v|) V^{-1/2} \bigg(\sum_{ x \leq n \leq x^3} \frac{\L^2(n)}{n^{2\s}} \bigg)^{1/2} \\ \notag
& \ll & T (|u|+|v|) V^{-1/2} \frac{x^{1/2-\s}}{(2\s-1)^{1/2}} \l^{1/2} x\ll T (|u|+|v|) V^{-1/2} x^{1/2-\s} \l T
\eea
(the estimate of the sum follows from the Prime Number Theorem). The result now follows on 
combining this estimate and  \eqref{summary}.
\end{proof}

\subsection{A Dirichlet Polynomial Calculation}
Our next goal is to show that the ch.f. of $ \sum_{n \leq x}  {\L(n)}n^{-(\s+it)}$ is essentially a Gaussian. 

\begin{lemma} \label{c.f. 2}  Let $1/2+4/\l x \leq \s \leq 2$,  and $ x \leq T^{1/(5N)}$ 
be sufficiently large
with $N$ an even integer. Then for  $|u|,|v|\leq V^{1/2}/100$,  we have
\bes  
\frac1T \int_0^T \, e\bigg(\vec{u} \cdot \sum_{n \leq x} \frac{\L(n)}{n^{\s+it}}V^{-1/2} \bigg) \, dt=e^{-2 \pi^2(u^2+v^2)}(1+E_4(u,v))+E_5(u,v),
\ees
where 
\[
E_4(u,v)\ll \frac{(|u|+|v|)^3}{V^{3/2}} +(u^2+v^2)(x^{1-2\s}(2\s-1) \l x+(2\s-1)^2)
\]
and 
\[
E_5(u,v) \ll  \frac{(6 \sqrt{2} \pi(|u|+|v|))^N}{(N/2)!}+T^{-1/3}.
\]
\end{lemma}
Before proving Lemma \ref{c.f. 2} we require several additional lemmas.  

\begin{lemma} \label{sound's lemma} Let $2 \leq x \leq T$. 
Also let $k$ be a natural number such that $x^k\leq T/\l T$.  Then for any complex numbers $a_p$ we have
\bes
\int_0^T \! \bigg| \sum_{p \leq x} a_p p^{-it} \bigg|^{2k} \, dt \ll T k!  \bigg(\sum_{p \leq x} |a_p|^2\bigg)^k.
\ees
\end{lemma}

\begin{proof}
This is due to Soundararajan in ~\cite{So}.
\end{proof}

The next lemma is a generalization of Montgomery and Vaughan's 
mean value theorem for Dirichlet polynomials ~\cite{MVMVT}. 

\begin{lemma} \label {MVMVT1}   Let $T \in \mathbb{R}$ and $m, k \in \mathbb{N}$. For any complex numbers $a_n$, $b_n$
\bes 
\begin{split}
\int_0^{T} \bigg( \sum_n a_n n^{-it} \bigg)^m \overline{\bigg(\sum_n b_n n^{-it}  \bigg)}^k \, dt
=& T \sum_n A_n \overline{B_n}\\
&+O\bigg(\bigg(\sum_n n |A_n|^2\bigg)^{1/2}\bigg(\sum_n n |B_n|^2\bigg)^{1/2}\bigg),
\end{split}
\ees
where 
\[
 A_n=\sum_{n_1\cdots n_m=n} a_{n_1}\cdots a_{n_m} \qquad\hbox{and}\qquad
  B_n=\sum_{n_1\cdots n_k=n} b_{n_1}\cdots b_{n_k}.
 \]
\end{lemma}

\begin{proof}
This is proved in K.M. Tsang's PhD thesis ~\cite{Ts}.
A proof may also be found in ~\cite{Ts2}.
\end{proof}
We now make an observation about the main term in Lemma \ref{MVMVT1}.  
Let $g(n)$ be a multiplicative function defined by $g(1)=1$ and 
$g(n)=\a_1 \theta_{p_1}+\cdots+\a_r \theta_{p_r}$, for $n=p_1^{\a_1}\cdots p_r^{\a_r}$ 
and indeterminates $\theta_{p_1},...,\theta_{p_r}$. 
By the unique factorization of the integers, and since for $m \in \mathbb Z$ we have 
\bes
\int_0^1 \! e(m\theta) \, d\theta=
\begin{cases} 
  1 & \mbox{if  } m = 0, \\
  0 & \mbox{if  } m \neq 0,
\end{cases}
\ees
we see that 
\be \label{clever observation}
\sum_n A_n \overline{B}_n= 
\int_0^1 \cdots \int_0^1 \! \bigg( \sum_n a_n e(g(n)) \bigg) 
\overline{\bigg( \sum_n b_n e(g(n)) \bigg)} \, \prod_{p}d\theta_p.
\ee

To state our next lemma write
\bes
S(\theta)=V^{-1/2}\sum_{p^n \leq x} \frac{\l p }{p^{n\s}}e(n \theta_p) 
\hspace{.3 in} \mbox{and} \hspace{.3 in} f(t)=V^{-1/2}\sum_{p^n \leq x} \frac{\l p}{p^{n(\s+it)}}.
\ees
 
\begin{lemma} \label{moments lem 11} Let  $1/2 \leq \s \leq 1 $, 
$m, k =0, 1, 2, \ldots$, and $e^{20} \leq x \leq T^{1/(5(m+k))}$. Then  
\be \label{moments 11}
\frac1T\int_0^{T} f(t)^m \overline {f(t)^k}\, dt= 
\int\limits_{\mathbb{T}^{\pi(x)}} \! S(\theta)^m  \overline{S(\theta)^k} \, d\theta  
+O\Big(V^{- (m+k)/2}T^{-1/3}(m!k!)^{1/2} \Big).
\ee
Moreover,
\bes
\frac1T \int_0^{T} \big| f(t) \big|^{2k} \,dt ,   \quad
\int\limits_{\mathbb{T}^{\pi(x)}} \! \big| S(\theta) \big|^{2k} \, d\theta 
\ll 18^k k! \ .
\ees
If $m=k=0$, then \eqref{moments 11}  holds without an error term.
\end{lemma}

\begin{proof} 
The last assertion of the lemma is obvious. 
We may therefore assume that at least one of $m, k \neq 0$.
Without loss of generality, we assume that $m\neq0$.
Applying Lemma \ref{MVMVT1} 
and \eqref{clever observation},  we find that
\be\label{moments 2} 
\begin{split}
V^{ (m+k)/2}\int_0^{T} f(t)^m \overline{f(t)^k}\, dt= &
V^{ (m+k)/2}T \int\limits_{\mathbb{T}^{\pi(x)}} \! S(\theta)^m  \overline{S(\theta)^k} \, d\theta \\
& +O\bigg(\bigg(\sum_n n |A_n|^2\bigg)^{1/2}\bigg(\sum_n n |B_n|^2\bigg)^{1/2}\bigg),
\end{split}
\ee 
where, for $m,k\neq0$,
\[
\ds A_n= \sum_{\substack{p_1^{n_1}\cdots p_m^{n_m} =n \\ p_j^{n_j} \leq x, j=1,2,...,m}} 
\frac{\l p_1 \cdots \l p_m}{(p_1^{n_1}\cdots p_m^{n_m})^{\s}} 
\quad\hbox{and}\quad 
 \ds B_n=\sum_{\substack{p_1^{n_1}\cdots p_{k}^{n_{k}} =n \\ p_j^{n_j} \leq x, j=1,2,...,k}} 
\frac{\l p_1 \cdots \l p_{k}}{(p_1^{n_1}\cdots p_{k}^{n_{k}})^{\s}}.
\] 
If $k=0$, then $B_1=1$ and $B_n=0$ for $n=1, 2, \ldots$.
By Cauchy's inequality,    
\bea \notag
\sum_{n \leq x^m} n |A_n|^2&=&\sum_{n \leq x^m} n 
\bigg(\sum_{\substack{p_1^{n_1}\cdots p_m^{n_m} =n \\ p_j^{n_j} \leq x, j=1,2,\ldots,m}} 
\frac{\l p_1 \cdots \l p_m}{(p_1^{n_1}\cdots p_m^{n_m})^{\s}}\bigg)^2 \\ \notag
& \leq& \sum_{n \leq x^m} n 
\bigg(\sum_{\substack{p_1^{n_1}\cdots p_m^{n_m} =n \\ p_j^{n_j} \leq x, j=1,2,\ldots,m}} \l^2 p_1 \cdots \l^2 p_m \bigg)
 \bigg(\sum_{\substack{p_1^{n_1}\cdots p_m^{n_m} =n \\ p_j^{n_j} \leq x, j=1,2,\ldots,m}} (p_1^{n_1}\cdots p_m^{n_m})^{-2\s}\bigg).\notag
\eea
Now, given $n, n_1, \ldots, n_m$, the equation $n=p_1^{n_1}\cdots p_m^{n_m}$ 
has at most $m^m$ solutions in $(p_1, \ldots , p_m)$.  Therefore,
\bes
\sum_{\substack{p_1^{n_1}\cdots p_m^{n_m} =n \\ p_j^{n_j} 
\leq x, j=1,2,\ldots,k}}(p_1^{n_1}\cdots p_m^{n_m})^{-2\s}\leq m^m
 \sum_{p_1,\ldots, p_m \leq x} \;\;\sum_{n_1, \ldots , n_m =1}^{\infty}  (p_1^{n_1}\cdots p_m^{n_m})^{-2\s}, 
\ees
which by  Mertens' theorem is
\bes
 \ll 6^m m! \bigg(\sum_{p \leq x} \frac{1}{p^{2\s}} \bigg)^m\ll (7 \l \l x)^m m!\ . 
\ees
Hence,
\bea \notag
\sum_{n \leq x^m} n |A_n|^2 &\ll& (7 \l \l x)^m m! \sum_{ n\leq x^m} 
\sum_{\substack{p_1^{n_1}\cdots p_k^{n_k} =n \\ p_j^{n_j} \leq x, j=1,2,...,m}} p_1^{n_1}\cdots p_m^{n_m} \l^2 p_1 \cdots \l^2 p_m\\ \notag
&=& (7 \l \l x)^m m! \bigg(\sum_{n\leq x} n \Lambda^2(n) \bigg)^m.
\eea
Since  \
$\sum_{n \leq x} n \L^2(n) \leq 2 x^2 \l x  $
\  and \ $7 \l \l x < \l x$ for $x>e^{20}$, we see that
 \bes
\sum_{n \leq x^m} n |A_n|^2 \ll m! ( 2x \l x)^{2m}.
\ees
Since $e^{20} \leq x \leq T^{1/(5(m+k))}$, we easily see that
$(2 x \l x )^{2m} < T^{2/3} $. Hence,
 \bes
\sum_{n \leq x^m} n |A_n|^2 \ll m! T^{2/3}.
\ees
Now, if $k=0$, \eqref{moments 11} follows since $\sum_{n \leq x^k} n |B_n|^2 =1$.  If $k \neq 0$, we 
similarly have 
\bes
\sum_{n \leq x^k} n |B_n|^2  \ll k! T^{2/3}.
\ees
Thus, the error term in \eqref{moments 2} is
\[
\ll  (m! k!)^{1/2} T^{2/3},
\]
and \eqref{moments 11} follows.

To prove the second assertion of the lemma we start with the observation that
\bes
\bigg|  \sum_{p^n \leq x} \frac{ \l p}{p^{n(\s+it)}} \bigg|^{2k} \leq
 9^k\bigg(\bigg| \sum_{p \leq x} \frac{ \l p}{p^{\s+it}} \bigg|^{2k} 
 +\bigg| \sum_{p^2 \leq x} \frac{ \l p}{p^{2(\s+it)}} \bigg|^{2k} +\frac{\z'}{\z}(3/2)^{2k}\bigg).
\ees
By Lemma \ref{sound's lemma}
 \bes
\frac1T \int_{0}^T \! \bigg|\sum_{p \leq x} \frac{ \l p}{p^{\s+it}}\bigg|^{2k} \! dt \ll  k!  (2V)^{k}.
\ees
Making the change of variable,  $u=2t$, we also see that
\bes
\frac1T \int_{0}^T \! \bigg|\sum_{p \leq x} \frac{ \l p}{p^{2\s+2it}}\bigg|^{2k} \! dt \ll  k!.  
\ees
It follows that
\[
\frac1T \int_0^{T} \big| f(t) \big|^{2k} \,dt \ll 18^k k!.
\]
To obtain the analogous bound for $\int\limits_{\mathbb{T}^{\pi(x)}} \! \big| S(\theta) \big|^{2k} \, d\theta$, 
we apply  \eqref{moments 11} 
and note that the error term is  $\ll 18^k k!$.
\end{proof}

Applying the preceding lemma, we can now prove 
\begin{lemma} \label{moments prop}
Let  $1/2+4/\l x \leq \s \leq 1$ and $ x \leq T^{1/(5N)}$ be sufficiently large with 
$N$ an even integer. Then
\be \label{prop}
\begin{split}
\frac1T \int_0^{T} \! e(\vec{u} \cdot f(t)) \, dt=  &
\int\limits_{\mathbb{T}^{\pi(x)}} \! e(\vec{u} \cdot  S(\theta)) \, d\theta
+O\bigg( \frac{(6 \sqrt{2} \pi(|u|+|v|))^N}{(N/2)!}\bigg)\\
&+O\bigg(T^{-1/3 }  \sum_{k=0}^{N-1} (6 \sqrt{2} \pi V^{-1/2}(u^2+v^2)^{1/2})^k\bigg).
\end{split}
\ee
\end{lemma}
\begin{proof}  
We begin by noting that for any complex number $z$,
\be \notag
u \Re z+v \Im z=\frac12(u-vi)z+\frac12(u+vi)\overline{z}.
\ee
 Define $C_1=\frac12(u-vi)$ and $C_2=\frac12(u+vi)$.  
 By expanding the exponential we see that the left-hand side of \eqref{prop} is equal to   
\be \label{star 1}
\begin{split}
\sum_{k=0}^{N-1}\frac{(2 \pi i)^k}{k!} 
\sum_{j=0}^k \binom{k}{j} C_1^j C_2^{k-j}\frac{1}{T}  \int_0^{T} \! f(t)^j  \overline{f(t)^{k-j} } \, dt
+ O\bigg(\frac{(2\pi)^N}{N!}(|u|+|v|)^N\frac{1}{T} \int_0^{T} \big| f(t) \big|^N \, dt  \bigg).
\end{split}
\ee
Similarly, expanding the main term on the right-hand side of the \eqref{prop}, we see that
\be \label{star 2}
\begin{split}
\sum_{k=0}^{N-1}\frac{(2 \pi i)^k}{k!} \sum_{j=0}^k \binom{k}{j} C_1^j C_2^{k-j} 
 \intl_{\mathbb{T}^{\pi(x)}} \! S(\theta)^j\overline{(S(\theta))}^{k-j} \, d\theta
+  O\bigg(\frac{(2\pi)^N}{N!}(|u|+|v|)^N 
\intl_{\mathbb{T}^{\pi(x)}}  \big|S(\theta) \big|^N \, d\theta  \bigg).
\end{split}
\ee
By Lemma \ref{moments lem 11}, 
\bes
\frac1T \int_0^{T} \big| f(t) \big|^N \,dt, 
\quad\intl_{\mathbb{T}^{\pi(x)}}  \big|S(\theta) \big|^N \, d\theta   \
\ll \  18^{N/2} (N/2)! \ .
\ees
Thus, the error terms in \eqref{star 1} and \eqref{star 2} are 
\bes
\ll \frac{(6 \sqrt{2} \pi(|u|+|v|))^N}{(N/2)!}\ .
\ees
Next we difference the main terms of \eqref{star 1} and \eqref{star 2} and
apply Lemma \ref{moments lem 11} to see that
 \bea \notag
&&\sum_{k=0}^{N-1}\frac{(2 \pi i)^k}{k!} 
\sum_{j=0}^k \binom{k}{j} C_1^j C_2^{k-j} \bigg(\frac{1}{T} \int_0^T\! f(t)^j\overline{(f(t))}^{k-j} \, dt-
\intl_{\mathbb{T}^{\pi(x)}} \! S(\theta)^j\overline{(S(\theta))}^{k-j} \, d\theta\bigg)\\ \notag
 && \ll  T^{-1/3 }  \sum_{k=0}^{N-1} \frac{(6 \sqrt{2} \pi V^{-1/2}(u^2+v^2)^{1/2})^k}{k!} 
 \sum_{j=0}^k \binom{k}{j} (j! (k-j)!)^{1/2}  \\ \notag
 && \ll T^{-1/3 }  \sum_{k=0}^{N-1} (6 \sqrt{2} \pi V^{-1/2}(u^2+v^2)^{1/2})^k.    
 \eea 
 The result now follows.
\end{proof}

Finally, we prove the following lemma, which, when combined with 
Lemma \ref{moments prop}, implies Lemma \ref{c.f. 2}.

\begin{lemma} \label{normal dist lem} Let $10 \leq x \leq T$, $(2\s-1) \l x \geq 1$
and $\ds |u|, |v| < V^{1/2}/100$. Then
\be \label{ch.f. S}
\int\limits_{\mathbb{T}^{\pi(x)}} \! e(\vec{u} \cdot S(\theta)) \,d\theta  \\
=e^{-2 \pi^2(u^2+v^2)} \big(1+ E_6(u,v)\big),
\ee
where
\bes
E_6 \ll   \frac{(|u|+|v|)^3}{V^{3/2}}   
  + (u^2+v^2)(x^{1-2\s}((2\s-1)\l x+1)) .
\ees 
\end{lemma}
\begin{proof}
We first note that the $\theta_p$ are independent variables, so
\bes
\int\limits_{\mathbb{T}^{\pi(x)}} \! e(\vec{u} \cdot S(\theta)) \, d\theta
=\prod_{p \leq x} \int_0^1 \! e\bigg(-\vec{u} \cdot \frac{\l p}{V^{1/2}} 
\sum_{n\leq \l_p x}   e(n\theta_p) p^{-n\s} \bigg) \, d\theta_p,
\ees
where $\l_p x$ denotes the logarithm of $x$ with respect to base $p$.
Expanding the exponential in the integrand on the right-hand side of the 
equation above and integrating the first three terms, we find that
\bes  
\begin{split}
\int\limits_{\mathbb{T}^{\pi(x)}} \! e(\vec{u} \cdot S(\theta)) \, d\theta
=&\prod_{p \leq x} \int_0^1 \, \sum_{\ell =0}^{\infty} 
\frac{\big(\vec{u} \cdot 2\pi i \l p \sum_{n\leq \l_p x} 
  e(n\theta_p){p^{-n\s}}\big)^{\ell}}{\ell! V^{\ell/2}} \, d\theta_p 
\\ \notag
=&\prod_{p \leq x} \Bigg(1- \frac{\pi^2(u^2+v^2) }{V}
\sum_{n \leq \l_p x} \frac{\l^2 p}{ p^{2n \s} } \\ \notag
&\qquad \qquad \qquad \qquad +\int_0^1 \, \sum_{\ell =3}^{\infty} 
\frac{\big(\vec{u} \cdot 2\pi i \l p \sum_{n\leq \l_p x} 
  e(n\theta_p){p^{-n\s}}\big)^{\ell}}{\ell! V^{\ell/2}} \, d\theta_p \Bigg).
\end{split}
\ees

Now write the right-hand side of this equation as $ \prod_{p \leq x} (1-M_p+R_p)$. 
Since we are assuming that $ |u|, |v| < V^{1/2}/100$, 
we have that 
\be \label{2 M bound}
M_p=\frac{\pi^2(u^2+v^2) }{V}\sum_{n < \l_p x}\frac{\l^2 p}{ p^{2n \s}  }<
\frac{\pi^2 (u^2+v^2) \l^2 p}{V(p^{2\s}-1)} < \frac{2 \pi^2}{100^2} <1/3.
\ee
Here we have used the fact that $\l^2 x/(x^{2\s}-1) < 1$ for $x\geq 2$. For $ |u|, |v| < V^{1/2}/100$
we also have that
\be \label{2 R bound}
\begin{split}
\AB  \int_0^1 \, \sum_{\ell =3}^{\infty} &
\frac{\Big(\vec{u} \cdot 2\pi i \l p \sum_{n\leq \l_p x} 
 e(n\theta_p)  p^{-n\s}\Big)^{\ell}}{\ell! V^{\ell/2}}  \AB \\
& \qquad \qquad \qquad \qquad \qquad \leq 
  \sum_{\ell=3}^{\infty} \frac{\bigg(2\pi (|u|+|v|) \l p \sum_{n \leq \l_p x}
 p^{-n\s} \bigg)^{\ell}}{\ell ! V^{\ell/2}} \\
 &  \qquad \qquad \qquad \qquad \qquad <
  \sum_{\ell=3}^{\infty} \frac{\Big(12\pi/100  
  \Big)^{\ell}}{\ell ! } <1/3.
  \end{split}
  \ee
It follows that we may expand the logarithm of $1-M_p+R_p$ in powers
of $-M_p+R_p$. 

Now, from the estimates in \eqref{2 M bound} and \eqref{2 R bound}, it is not difficult to see that
\[
|R_p| \ll \frac{(|u|+|v|)^3\l^3 p}{V^{3/2}p^{3\s}} \qquad \hbox{and} \qquad 
 |M_p| \ll \frac{(u^2+v^2)\l^2 p}{Vp^{2\s}}.
 \]
Hence, 
\bea \notag
\prod_{p \leq x}(1-M_p+R_p)&=&\prod_{p \leq x} \exp(\l(1-M_p+R_p)) \\ \notag
&=&\exp\bigg(\sum_{p \leq x}(-M_p+R_p+O((M_p+R_p)^2))\bigg) \\ \notag
&=& \exp\bigg(\sum_{p \leq x} -M_p+O\bigg(\frac{(|u|+|v|)^3\l^3 p}{V^{3/2} p^{3\s}}\bigg)\bigg).  
\eea
Thus,
\be  \label{expansion}
\begin{split}
\int\limits_{\mathbb{T}^{\pi(x)}} \! e(\vec{u} \cdot S(\theta)) \, d\theta
=&\prod_{p \leq x}  
\exp\bigg(-\frac{\pi^2 (u^2+v^2)}{V} \sum_{n \leq \l_p x}\frac{\l^2 p}{ p^{2 n\s} } 
+ O\Big(\frac{(|u|+|v|)^3}{V^{3/2} }\frac{\l^3 p}{p^{3\s}}\Big)\bigg)\\
 =&\exp\bigg(-\frac{\pi^2 (u^2+v^2) }{V}\sum_{p^n \leq x} \frac{\l^2 p}{p^{2n\s}}  \bigg)
 \bigg(1+O\bigg( \frac{(|u|+|v|)^3}{V^{3/2}}\bigg)\bigg).
\end{split}
\ee 
Finally, a short calculation using the Prime Number Theorem reveals that
\be \notag
\begin{split}
V^{-1}\sum_{p^n \leq x} \frac{\l^2 p}{p^{2n\s}} = &2\cdot\frac{\sum_{m=2}^{\infty} \frac{\Lambda^2(m)}{m^{2\s}}+
O\Big(\sum_{m > x}^{\infty} \frac{\Lambda^2(m)}{m^{2\s}} \Big)}{\sum_{m=2}^{\infty} \frac{\Lambda^2(m)}{m^{2\s}}}\\
=&2+O\bigg((2\s-1)^2 \cdot \frac{x^{1-2\s}}{(2\s-1)^2}((2\s-1)\l x+1)\bigg).
\end{split}
\ee
Using this estimate on the last line of \eqref{expansion}, we obtain \eqref{ch.f. S}.
\end{proof}

\subsection{The Proof of Theorem \ref{theorem 2}}
By Lemma \ref{c.f. 1} and Lemma \ref{c.f. 2} it follows that, for $|u|, |v| < V^{1/2}/100$,
\bes
\frac1T \int_0^T \! e\bigg(-\vec{u} \cdot \frac{\z'}{\z}(\s+it) V^{-1/2} \bigg) \, dt= e^{-2 \pi^2(u^2+v^2)}\big(1+E_4(u,v)\big)+E_5(u,v)+ E_1
\ees
where 
\bes
\begin{split}
 & E_1 \ll (|u|+|v|)V^{-1/2}x^{(1/2-\s)/2} \l T+T^{-(\s-1/2)/3}\frac{\l T}{\l x},  \\
 & E_4(u,v)\ll \frac{(|u|+|v|)^3}{V^{3/2}} +(u^2+v^2)(x^{1-2\s}((2\s-1) \l x+1)), \\
 & E_5(u,v) \ll  \frac{(6 \sqrt{2} \pi(|u|+|v|))^N}{(N/2)!}+T^{-1/3}.
\end{split}
\ees
Now take $N=2 \left \lfloor \psi(T)/(800 \l \psi(T)) \right \rfloor$, 
where $\psi(T)$ is any function that tends to infinity with $T$ 
and is also $o(\l T)$. Also, let $x=T^{1/(5N)}$, 
$\s=1/2+\psi(T)/(2\l T)$, and 
$|u|, |v| < \min\(V^{1/2}, 1/5 N^{1/2}\)/100$. 
We   note that $\s> 1/2+4/\l x$ and $2V=1/(2\s-1)^2+O(1)$.  
With these choices we find that 
\[
 E_4(u,v)\ll \frac{(|u|+|v|)^3}{V^{3/2}} +\frac{(u^2+v^2)}{\psi(T)^{10}},
\] 
and that 
\[
E_1, E_5(u,v) \ll  \psi(T)^{-10}.
\]  
This gives Theorem \ref{theorem 2}.
\subsection{An Upper Bound for the Characteristic Function}
Before proving Theorem \ref{theorem 1.1}, we
 prove a lemma
that will be needed in the proof of Theorem \ref{theorem disk 1}.
This lemma enables us to get an upper bound on 
the ch.f. of $\z'/\z(\s+it)V^{-1/2}$ for $|u|$ or $|v|$
larger than $\Omega=e^{-10} \min\( V^{1/2}, (\psi(T)/\l \psi(T))^{1/2} \)$.
One may compare this to Theorem \ref{theorem 2}, where we obtained
an asymptotic formula for the ch.f., but only
for $|u|,|v|\leq \Omega$. 

\begin{lemma} \label{tail end}
Let  $\psi(T)=(2\s-1)\l T$ be any positive function that tends to infinity with $T$ 
and is also $o(\l T)$. 
There exists a positive constant $K$ such that if $T$ is sufficiently large, 
\[
\tOmega = \min\bigg(e^{-10} \Big(\frac{\psi(T)}{ \l \psi(T))}\Big)^{1/2},  Ke^{\s/(2\s-1)} \Big) \bigg),
\]
and $|u|,|v| \leq \tOmega$, then we have that
\bes 
\frac1T \int_0^T \, e\bigg(\vec{u} \cdot\frac{\z'}{\z}(\s+it)V^{-1/2}\bigg) \, dt
\ll e^{-c(u^2+v^2)}+ \psi(T)^{-10}.
\ees
Here $c$ is a positive absolute constant.
\end{lemma}
\begin{proof}
As in the proof of Theorem \ref{theorem 2} we let 
$N=2 \left \lfloor \psi(T)/(800 \l \psi(T) )\right \rfloor$. Also, let $x=T^{1/(5N)}$ and 
$\s=1/2+\psi(T)/(2 \l T)$. It then follows that    $\s>1/2+4/\l x$.
With these choices we find from Lemma \ref{c.f. 1} that
\[
\frac1T \int_0^T \, e\bigg(-\vec{u} \cdot\frac{\z'}{\z}(\s+it)V^{-1/2}\bigg) \, dt
=\frac1T \int_0^T \, e\bigg(\vec u \cdot \sum_{n \leq x} \frac{\L(n)}{n^{\s+it}}V^{-1/2}\bigg) \, dt+\psi(T)^{-10}.
\]
Applying Lemma \ref{moments prop}, we have that for $|u|,|v| \leq \tOmega$, the right-hand side
of this equation equals
\[
\intl_{\mathbb{T}^{\pi(x)}} e(\vec u \cdot S(\theta)) \, d\theta+\psi(T)^{-10}.
\]
Hence, to prove the lemma it suffices to show that
\bes
\intl_{\mathbb{T}^{\pi(x)}} e(\vec u \cdot S(\theta)) \, d\theta \ll e^{-c(u^2+v^2)},
\ees
for $|u|,|v| \leq \tOmega$. 

If $|u|$ and $|v|$ are less than $V^{1/2}/100$
we are done, by Lemma \ref{normal dist lem}. Thus, we may assume that 
 $|u|$ or $|v|$ is
greater than or equal to $V^{1/2}/100$.
First observe that for any $D \geq 2$,
\be \label{initial obs}
\AB \intl_{\mathbb{T}^{\pi(x)}} e\(\vec u \cdot S(\theta) \) \, d\theta \AB 
\leq \AB \prod_{D \leq p \leq x} \int_0^1 e\(\vec u \cdot S(\theta_p) \) \, d\theta_p \AB.
\ee
Now take
\[
D=C^{1/\s} \( \frac{|u|+|v|}{\sqrt{ V}}\)^{1/\s}\(\l \(C \frac{|u|+|v|}{\sqrt V}\)\)^{1/\s},
\]
where $C \geq 200$ is an absolute constant to be chosen later.  We   
expand the exponential on the right-hand side of \eqref{initial obs} and 
integrate term-by-term, as in the
proof of Lemma \ref{normal dist lem}. We thus find that
\be
\begin{split}
\int_0^1 e\(\vec u \cdot S(\theta_p) \) \, d\theta_p=& 1-\frac{\pi^2(u^2+v^2)}{V} 
\sum_{n \leq \l_p x}\frac{\l^2 p}{ p^{2n \s} } \\
&+\int_0^1 \, \sum_{\ell =3}^{\infty} 
\frac{\big(\vec{u} \cdot 2\pi i \l p \sum_{n\leq \l_p x} 
  e(n\theta_p) p^{-n\s}\big)^{\ell}}{\ell! V^{\ell/2}} ,
\end{split}
\ee
where $\l_p x$ means the logarithm of $x$
with respect to base $p$. Next notice  that 
$\l x/x^{\s}$ is decreasing for $x\geq 10$. Thus, for $p\geq D$ we have that
$\l p/p^{\s} \leq \l D/D^{\s}$. That is,
\bes 
\frac{\l p}{p^{\s}} \leq 
\frac{
\sqrt{V} 
\l\(C \frac{ (|u|+|v|)}{\sqrt V}
\l\(C\frac{(|u|+|v|)}{\sqrt V}\)\)}
{C \s  (|u|+|v|) \l \( C\frac{ (|u|+|v|)}{\sqrt V} \)  }. 
\ees
From this we  easily see that
\be \label{D bd}
(|u|+|v|) \frac{\l p}{\sqrt{V} p^{\s}} \leq \frac{2}{C} \cdot
\frac{\l \( C\frac{ (|u|+|v|)}{\sqrt V} \)+\l \l \( C\frac{ (|u|+|v|)}{\sqrt V} \)}{\l \( C\frac{ (|u|+|v|)}{\sqrt V} \)}
\leq \frac{4}{C}.
\ee
Consequently,
\bes \label{Mp bd}
\begin{split}
\frac{\pi^2(u^2+v^2) }{V}\sum_{n \leq \l_p x}\frac{\l^2 p}{ p^{2n \s}  }\leq &
\frac{\pi^2 (|u|+|v|)^2 \l^2 p}{V(p^{2\s}-1)}\\
\leq & \frac{32 \pi^2}{C^2}.
\end{split}
\ees
Next note that
\bes
\begin{split}
\Bigg| \int_0^1 \, \sum_{\ell =3}^{\infty} 
\frac{\Big(\vec{u} \cdot 2\pi i \l p \sum_{n\leq \l_p x} 
 e(n\theta_p)p^{-n\s}\Big)^{\ell}}{\ell! V^{\ell/2}} \Bigg| 
 &  \leq   \sum_{\ell=3}^{\infty} \frac{\bigg(2\pi (|u|+|v|) \l p \sum_{n \leq \l_p x}
 p^{-n\s} \bigg)^{\ell}}{\ell ! V^{\ell/2}} \\
 &  \leq  \sum_{\ell=3}^{\infty} \frac{\Big(8\pi (|u|+|v|)  \l p 
  \Big)^{\ell}}{\ell ! (\sqrt{V}p^{\s})^{\ell}}.
  \end{split}
  \ees
  By \eqref{D bd} this is
\bes
  \begin{split}
 \leq  \sum_{\ell=3}^{\infty} \frac{(32\pi /C 
  )^{\ell}}{\ell !}
  \leq \frac{e^{32\pi/C}(32\pi)^3}{C^3}.
\end{split} 
\ees
Hence, if $p\geq D$ we may expand the logarithm of 
\bes
\begin{split}
1-\frac{\pi^2(u^2+v^2)}{V} \sum_{n \leq \l_p x}\frac{\l^2 p}{ p^{2n \s}  } 
+\int_0^1 \, \sum_{\ell =3}^{\infty} 
\frac{\Big(\vec{u} \cdot 2\pi i \l p \sum_{n\leq \l_p x} 
e(n\theta_p)p^{-n\s}\Big)^{\ell}}{\ell! V^{\ell/2}} d\theta
\end{split}
\ees
whenever $C$ is sufficiently large. Write this 
as $1-M_p+R_p$. Thus, we obtain that
\be \label{taylor exp}
\prod_{D \leq p \leq x} \int_0^1 e\(\vec u \cdot S(\theta) \) \, d\theta
= \exp\bigg(\sum_{D \leq p \leq x}\big(-M_p+R_p+O(M_p^2+R_p^2)\big) \bigg).
\ee
  
   Since $M_p \ll (|u|+|v|)^2\l^2 p/(Vp^{2\s})$ and $R_p \ll (|u|+|v|)^3\l^3 p/(V^{3/2}p^{3\s})$,
   by \eqref{D bd} we have
  \bes 
  \begin{split}
  R_p+M_p^2+R_p^2 \ll& \frac{(|u|+|v|)^3\l^3 p}{V^{3/2}p^{3\s}}
  +\frac{(|u|+|v|)^4\l^4 p}{V^{2}p^{4\s}}+\frac{(|u|+|v|)^6\l^6 p}{V^{3}p^{6\s}}\\
  \ll& \frac{(|u|+|v|)^2 \l^2 p}{CVp^{2\s}} \ll \frac{(u^2+v^2) \l^2 p}{CVp^{2\s}}.
  \end{split}
  \ees
So, in particular, for $C$ large enough it follows that in \eqref{taylor exp} 
the term $R_p+O(M_p^2+R_p^2)$ is $\leq
(\pi^2/2)\cdot((u^2+v^2) \l^2 p/(Vp^{2\s}))$. Thus,
  \be \label{last bd 2}
  \begin{split}
  \prod_{D \leq p \leq x} \int_0^1 e\(\vec u \cdot S(\theta) \) & \, d\theta  
  \ll \exp\bigg(-\sum_{D \leq p \leq x}M_p+\frac{\pi^2}{2}\frac{(u^2+v^2) }{V} 
\sum_{D \leq p \leq x}\frac{ \l^2 p}{p^{2\s}} \bigg).
  \end{split}
  \ee
We now choose  $C$ to be large enough so that this holds, $|R_p| <1/3$, and $|M_p|<1/3$.

Next observe that 
  $M_p \geq \pi^2 (u^2+v^2)\l^2 p/(Vp^{2\s})$ so that
  \bes 
  \sum_{D \leq p \leq x} M_p  \geq  \frac{\pi^2 (u^2+v^2)}{V} \sum_{D \leq p \leq x} \frac{\l^2 p}{p^{2\s}}.
  \ees
 Applying this in \eqref{last bd 2}, we have
\be \label{collect ests}
\begin{split}
\prod_{D \leq p \leq x} \int_0^1 e\(\vec u \cdot S(\theta) \) \, d\theta \ll & 
\exp\bigg(- \frac{\pi^2}{2}(u^2+v^2)V^{-1}\sum_{D \leq p \leq x} \frac{\l^2 p}{ p^{2 \s}} \bigg).
\end{split}
\ee
By the Prime Number Theorem, there is an absolute constant $c_1>0$ such that
\be \label{final sum}
\begin{split}
\sum_{D \leq p \leq x} \frac{\l^2 p}{ p^{2 \s}}=&
\frac{1}{(2\s-1)^2}
 \Big(D^{1-2\s}((2\s-1)\l D +1)+O\Big((2\s-1) e^{-c_1\sqrt{\l D}}\Big)
\\
& \qquad \qquad \qquad -x^{1-2\s}((2\s-1)\l x  +1)
+O\Big((2\s-1)e^{-c_1\sqrt{\l x}}\Big)\Big).
\end{split}
\ee
This follows from the Prime Number Theorem with error term.

 By our choice of $x$ we note that
$x^{1-2\s}((2\s-1)\l x +1)=o(1)$. Also, both error terms are $o(1)$. 
Next observe that there is a positive absolute constant $K$ such that
whenever $|u|,|v|\leq \tOmega$, we have
\bes
\begin{split}
D\leq  & \(\frac{2KC}{V^{1/2}}\)^{1/\s} e^{1/(2\s-1)} \(\l\( \frac{2KC}{V^{1/2}}e^{\s/(2\s-1)}\) \)^{1/\s} \\
\leq & \(\frac{2KC}{V^{1/2}}\)^{1/\s} e^{1/(2\s-1)}\(\l e^{\s/(2\s-1)} \)^{1/\s} \\
\leq &  \(\frac{2KC\s}{(2\s-1)V^{1/2}}\)^{1/\s} e^{1/(2\s-1)} 
\leq e^{1/(2\s-1)},\\
\end{split}
\ees
where the last estimate follows from the Prime Number Theorem.
Hence, $D^{1-2\s} \geq e^{-1}$ and
the right-hand side of \eqref{final sum} is $\gg 1/(2\s-1)^2 \gg V$.  Combining this with
\eqref{initial obs} and \eqref{collect ests} completes the proof.
\end{proof}
\section{The Rate of Convergence to the
Normal Distribution}
\subsection{Beurling-Selberg Functions}
The next two lemmas
state properties of Beurling-Selberg functions.
These functions allow us to obtain bounds on the
rate of convergence
of the distribution of $\z'/\z(\s+it)V^{-1/2}$
to the normal distribution.
\begin{lemma} \label{magic functions 2}
Let $\d$ be a positive real number, let $a,b \in \R$, and let
$z=x+iy$. There exists an entire function
$F(z)$ with the following properties:
\bi
 \item[i)] $0\leq \(F(x)- \mathbf{1}_{[a, b]}(x)\) \ll \displaystyle \frac{\sin^2(\pi \d (x-a))}{(\pi \d (x-a))^2}
 +\frac{\sin^2(\pi \d (x-b))}{(\pi \d (x-b))^2}$;
 \item[ii)] $\ds \int_{-\infty}^{\infty} (F(x)-\mathbf{1}_{[a,b]}(x)) \, dx \ll 1/\d $;
 \item[iii)] $\widehat F(\xi)=0$, for $\xi \in \R$ with $|\xi| \geq \d$;
 \item[iv)] $\widehat F(\xi) \ll |b-a|+1/\d$, for $\xi \in \R$.
\ei
\end{lemma}
\begin{proof}
Property $i)$ follows from
Lemma 5 of ~\cite{Va}, and property $ii)$ follows from property $i)$.
Property $iii)$ follows from ~\cite{Montg} 
(see the argument directly after the proof of Lemma 5). 
To obtain $iv)$, note that
that by $i)$  the $L^1$ norm of $F(x)$ is $\ll |b-a|+1/\d$. 
\end{proof}
\begin{lemma} \label{magic functions}
Let $r$ and $\d$ be positive real numbers
with $r \d \geq 1$. Also
let $\mathbf{z} \in \C^2$ and
$\vec x=(x_1, x_2) \in \R^2$.
Then there exist entire functions
$F_+(\mathbf{z})$ and $F_-(\mathbf{z})$,
 with the following properties:
\begin{itemize}
	\item[i)] $F_-(\vec{x}) \leq \mathbf{1}_{[0, r]}(|\vec{x}|) \leq F_+(\vec{x})$;
	\item[ii)] $\intl_{\R^2} \(F_{+}(\vec{x}))-F_{-}(\vec{x})) \) \, d \vec x \ll r/\delta$;
	\item[iii)] $\widehat F_{\pm}(\vec \xi)=0$, for $\vec \xi \in \R^2$ with $|\vec \xi|\geq \delta$;
	\item[iv)] $\widehat F_{\pm}(\vec \xi)\ll r^2$, for $\vec \xi \in \R^2$,
\end{itemize}
where $|\vec x|=\sqrt{x_1^2+x_2^2}$, 
and $d \vec x=dx_1 dx_2$.
\end{lemma} 
\begin{proof}
Properties $i)$ and $ii)$ follow from Theorem 3 of ~\cite{Ho}.  
By the same 
theorem, $F_{+}(z)$  and $F_{-}(z)$ are of
exponential type at most $2\pi \d$ (for the definition of exponential type see
 ~\cite{Ho}). Thus
$iii)$ follows by the Paley-Wiener Theorem  (see
Chapter III, Theorem 4.9 of ~\cite{St}).  
Finally, to obtain $iv)$, note that by $i)$ and $ii)$
the $L^1$ norm of $F_{\pm}(\vec x)$ is $ \ll r^2+r/\d
\ll r^2$, since $r \delta \geq 1$.
\end{proof}
\subsection{The Proof of Theorem \ref{theorem 1.1}}
Let $\psi(T)$, $\bOmega$ and $R$ be as in the statement of  Theorem \ref{theorem 1.1}.
Also, let $\vec u=(u,v)\in \R^2$, $d\vec u=dudv$, $\tOmega$ be as in Lemma \ref{tail end},
and $\Omega$ be as in Theorem \ref{theorem 2} .  
We take $F$ to be the analytic function from Lemma \ref{magic functions 2}
that approximates $\mathbf 1_{[a, b]}(x)$ along the real axis,
and $G$ to be the one that approximates $\mathbf 1_{[c, d]}(x)$.
We set $\delta=\tOmega$ in both functions. Also, let
\[
H(x)=\frac{\sin^2 (\pi \tOmega x)}{ (\pi \tOmega x)^2}.
\]

By property $i)$ of Lemma \ref{magic functions 2} we have
\be \label{indicator}
\begin{split}
F(x) G(y)=  \mathbf{1}_{R}(x,y) 
+O\Big( H(x-a)+H(x-b)+H(y-c)+H(y-d) \Big).
\end{split}
\ee
Now, note that
\be\label{thm 1 star}
\begin{split}
\frac1T \meas\bigg\{t \in[0, T] : \frac{\z'}{\z}(\s+it)  & V^{-1/2} \in R \bigg\} 
 =\frac1T \int_0^T \mathbf{1}_{R}\bigg(\frac{\z'}{\z}(\s+it) V^{-1/2}\bigg) \, dt.
\end{split}
\ee
By \eqref{indicator} the right-hand side of \eqref{thm 1 star} equals
\be \label{meas}
\frac1T \int_0^T F\bigg(\Re \frac{\z'}{\z}(\s+it) V^{-1/2}\bigg) G \bigg(\Im \frac{\z'}{\z}(\s+it) V^{-1/2}\bigg) \, dt
\ee
plus an error that is 
\bes
\begin{split}
&\ll  \frac1T \int_0^T \! H\bigg(\Re \frac{\z'}{\z}(\s+it) V^{-1/2}-a\bigg) \, dt 
+\frac1T\int_0^T \! H\bigg(\Re \frac{\z'}{\z}(\s+it) V^{-1/2}-b\bigg) \, dt \\ 
+&\frac1T\int_0^T \! H\bigg(\Im \frac{\z'}{\z}(\s+it) V^{-1/2}-c\bigg) \, dt
+\frac1T\int_0^T \! H\bigg(\Im \frac{\z'}{\z}(\s+it) V^{-1/2}-d\bigg) \, dt.
\end{split}
\ees
Since,
\bes
H(x)=\frac{2(1-\cos(2\pi \tOmega x))}{(2\pi \tOmega x)^2}
=\frac{2}{\tOmega^2} \int_0^{\tOmega} \! (\tOmega-u) \cos(2\pi x u) \, du,
\ees
the first term in the error above is
\bes
\ll  \frac{1}{\tOmega^2} \Re \int_0^{\tOmega} \! (\tOmega-u)  e(-a) \frac1T \int_0^T \! e\bigg(u \Re \frac{\z'}{\z}(\s+it) V^{-1/2} \bigg) \, dt du.
\ees
In the inner integral $|u| \leq \tOmega$, so
we may apply Lemma \ref{tail end} with $v=0$ to see that this is
\bes
\ll \frac{1}{\tOmega^2} \int_0^{\tOmega} \! (\tOmega-u)  \big(e^{-c u^2}+\psi(T)^{-10}\big) du  \ll 1/\tOmega.
\ees
Clearly, the other error terms can be bounded similarly.  
Hence, upon applying Fourier inversion to \eqref{meas}
we have that the left-hand side of \eqref{thm 1 star} equals
\bes 
 \intl_{\R^2} \, \widehat{F}(u)
  \widehat{G}(v) \, \frac1T \int_0^T e\(\vec u \cdot \frac{\z'}{\z}(\s+it) V^{-1/2}\) \, dt d\vec u+O(1/\tOmega).
\ees

By property $iii)$ of Lemma \ref{magic functions 2} the integral equals
\[
\int_{-\tOmega}^{\tOmega} \int_{-\tOmega}^{\tOmega} \, \widehat{F}(u)
  \widehat{G}(v) \frac1T \int_0^T e\(\vec u \cdot \frac{\z'}{\z}(\s+it) V^{-1/2}\) \, dt d\vec u.
\]
We now apply Theorem \ref{theorem 2} and Lemma \ref{tail end}
 to see that this is
\be \label{main int}
\begin{split}
 & \int_{-\Omega}^{\Omega} \int_{-\Omega}^{\Omega} \, \widehat{F}(u)
  \widehat{G}(v) (e^{-2\pi(u^2+v^2)}(1+\mathcal E_A(u,v))+\E_B) \, d\vec u \\
 & + O\bigg(\intl_{[-\tOmega, \tOmega]^2 \setminus [-\Omega, \Omega]^2}
 \! |\widehat{F}(u)|
  |\widehat{G}(v)| (e^{-c(u^2+v^2)}+\psi(T)^{-10}) \, d\vec u \bigg),
 \end{split}
 \ee
 where $\E_A(u,v)$ and $\E_B$ are as in Theorem \ref{theorem 2}.
We first estimate the integral in the $O$-term. Note that we are
 assuming that $(2\s-1)=o(1)$. This implies that 
 $\bOmega \leq \tOmega$, so that $|b-a|, |d-c| \geq \tOmega^{-1}$.
 Hence, by property $iv)$ of Lemma \ref{magic functions 2},
 the integral in the $O$-term is 
 \bes
 \begin{split}
 \ll& \intl_{[-\tOmega, \tOmega]^2 \setminus [-\Omega, \Omega]^2} 
 |b-a||d-c| (e^{-c (u^2+v^2)}+\psi(T)^{-10}) d\vec u \\
 \ll &
 |b-a||c-d| \bigg(\bigg(\int_{\Omega}^{\infty} \! e^{-c u^2} \, du\bigg)^2
 +\frac{\tOmega^2}{\psi(T)^{10}}\bigg).
 \end{split}
 \ees
 Since $\bOmega \ll \Omega^2$ and $\tOmega \leq \psi(T)^{1/2}$,
 this is easily 
 seen to be $\ll |b-a||c-d|/\bOmega$.

 We now
 write the first integral in \eqref{main int} as $I_1+I_2+I_3$,
 where $I_1$ is the integral of $\widehat{F}(u)
  \widehat{G}(v) e^{-2\pi(u^2+v^2)}$ over $[-\Omega, \Omega]^2$, $I_2$
  is the integral of 
  $  \widehat{F}(u)
  \widehat{G}(v)e^{-2\pi(u^2+v^2)}\E_A(u,v)$
  over $[-\Omega, \Omega]^2$, and
  $I_3$ is the rest.  
  Then to prove Theorem \ref{theorem 1.1} it suffices to show that
  \[
  I_1+I_2+I_3=\int_a^b \int_c^d e^{-(x^2+y^2)/2} dx dy +O((|b-a||d-c|+1)/\bOmega).
  \]
  By Theorem \ref{theorem 2},  $\E_A(u,v)\ll (|u|+|v|)^3/V^{3/2}+(u^2+v^2)/\psi(T)^{10}$ 
and $\E_B \ll  \psi(T)^{-10}$. As in our treatment of the  $O$-term 
in \eqref{main int}, we apply  property $iv)$ of Lemma \ref{magic functions 2} and find that
\be \label{I2 thm 1}
I_2 \ll |b-a||d-c|/V \ll |b-a||d-c|/\bOmega
\ee
and
\be \label{I3 thm 1}
I_3 \ll |b-a||d-c|/\psi(T)^9 \ll|b-a||d-c|/\bOmega.
\ee

To estimate $I_1$, we
extend the integral to all of $\R^2$ with a small
error that is easily seen to be $\ll |b-a||d-c|/\bOmega$. 
Next we apply Plancherel's theorem to see that
\bes 
\begin{split}
\intl_{\R^2} \widehat{F}(u)
  \widehat{G}(v) e^{-2\pi(u^2+v^2)} \, d\vec u=&
 \bigg(\intl_{\R} \widehat{F}(u)  e^{-2\pi u^2} du\bigg) 
 \bigg(\intl_{\R} \widehat{G}(v)  e^{-2\pi v^2} dv\bigg) \\
= &\frac{1}{2\pi} \bigg(\intl_{\R} F(x)  e^{-x^2/2} dx\bigg)\bigg(\intl_{\R} G(y)  e^{-y^2/2} dy\bigg).
\end{split}
\ees
By property $ii)$ of Lemma \ref{magic functions 2}
\[
\intl_{\R} F(x)  e^{-x^2/2} dx=\int_{a}^b  e^{-x^2/2} dx+O(1/\tOmega).
\]
An analogous result holds  for $G$, so we have 
\bes
\begin{split}
I_1=& \frac{1}{2\pi} \(\int_a^b e^{-x^2/2} dx +O\( \frac{|b-a||d-c|+1}{\bOmega} \)\) 
 \(\int_c^d e^{-y^2/2} dy +O\(\frac{|b-a||d-c|+1}{\bOmega}\)\)\\
=&  \frac{1}{2\pi} \int_a^b\int_c^d e^{-(x^2+y^2)/2} dx dy +O((|b-a||d-c|+1)/\bOmega).
\end{split}
\ees
This combined with \eqref{I2 thm 1} and \eqref{I3 thm 1} yields
\[
I_1+I_2+I_3=\int_a^b\int_c^d e^{-(x^2+y^2)/2} dx dy +O((|b-a||c-d|+1)/\bOmega).
\]
\subsection{The Proof of Theorem \ref{theorem disk 1}}
Let  $\psi(T)$, $\bOmega$, and $r$ be as in the statement of  Theorem \ref{theorem disk 1}.
Also, let $\vec u=(u,v)\in \R^2$, $d\vec u=dudv$, $\tOmega$ be as in Lemma \ref{tail end},
and let $\Omega$ be as in Theorem \ref{theorem 2}.
We also let $D_{r_1}(0)$ denote the disk of radius $r_1$ centered at the origin.
We now
consider $F_+(\mathbf z)$ from 
Lemma \eqref{magic functions} with $\delta=\tOmega$ 
(note $\bOmega \leq \tOmega$ so $\tOmega r \geq 1$ ).  

By Fourier inversion
\bes
\begin{split}
  &\frac1T \int_0^T \! F_+\( \Re \zpz(\s+it)V^{-1/2}, \Im \zpz(\s+it)V^{-1/2}\) \, dt \\
  = &\intl_{\R^2} \widehat F_{+}(\vec u) \, \, \frac1T \int_0^T e\(\vec u \cdot  \zpz(\s+it)V^{-1/2}\) \, dt \, d\vec u.
\end{split}
\ees
By property $iii)$ of Lemma \ref{magic functions}, this is
\bes
=\intl_{ D_{\tOmega}(0)} \widehat F_+(\vec u) \, \, \frac1T \int_0^T e\(\vec u \cdot  \zpz(\s+it)V^{-1/2}\) \, dt d\vec u,
\ees
and by Theorem \ref{theorem 2} and Lemma \ref{tail end}, this equals
\bes
\begin{split}
 \intl_{ D_{\Omega}(0)} \widehat F_+(\vec u) \, 
\Big(e^{-2 \pi^2(u^2+v^2)}(1+\E_A & (u,v))+ \E_B\Big) d\vec u\\
+&O\bigg(\intl_{\mathcal A_{\Omega, \tOmega(0)}} |\widehat F_+(\vec u)|
(e^{-c(u^2+v^2)}+\psi(T)^{-10}) d\vec u \bigg) ,
\end{split}
\ees
where $A_{\Omega, \tOmega}(0)$ is the annulus with radii $\Omega$ and $\tOmega$ centred at the origin.
By property $iv)$ of Lemma \ref{magic functions},  the integral over the annulus is
\[
\ll r^2 \intl_{\mathcal A_{\Omega, \tOmega(0)}} \(e^{-c(u^2+v^2)}+\psi(T)^{-10}\) d\vec u\ll r^2/\bOmega.
\]

We now write
\bes
\intl_{ D_{\Omega}(0)} \widehat F_+(\vec u) \, 
\(e^{-2 \pi^2(u^2+v^2)}(1+\E_A(u,v))+\E_B)\) d\vec u=I_1+I_2+I_3,
\ees
where $I_1$ is the integral of 
$\widehat F_+(\vec u)e^{-2 \pi^2(u^2+v^2)}$
 over $ D_{\Omega}(0)$,  $I_2$ is the integral of $\\ \widehat F_+(\vec u) e^{-2 \pi^2(u^2+v^2)}\E_A(u,v)$ over $ D_{\Omega}(0)$, 
and $I_3$ is the rest. By Theorem \ref{theorem 2}, $\E_A(u,v)\ll (|u|+|v|)^3/V^{3/2}+(u^2+v^2)/\psi(T)^{10}$ 
and $\E_B \ll \psi(T)^{-10} $. We first estimate $I_3$. By property $iv)$ of Lemma \ref{magic functions}
\bes
I_3 \ll r^2\tOmega^2/\psi(T)^{10}\ll r^2/\bOmega.
\ees
Similarly, we have
\bes
I_2 \ll r^2/\bOmega +r^2/\psi(T)^{10}\ll r^2/\bOmega.
\ees
To estimate $I_1$, we note that by property $iv)$ of Lemma \ref{magic functions}
\[
\intl_{ D_{\Omega}(0)} \widehat F_+(\vec u) e^{-2 \pi^2(u^2+v^2)} \, d\vec u
 =\intl_{\R^2} \widehat F_+(\vec u) e^{-2 \pi^2(u^2+v^2)} \, d\vec u+O(r^2/ \bOmega).
\]
By Plancherel's Theorem the integral on the right-hand side equals
\bes
\frac{1}{2\pi}\intl_{\R^2} F_+(\vec{x}) e^{-(x_1^2+x_2^2)/2} d\vec x.
\ees

Collecting our estimates, we   have that
\be\label{fplus id}
\begin{split}
\frac1T \int_0^T \! F_+\bigg( \Re \zpz(\s+it)V^{-1/2},& \Im \zpz(\s+it)V^{-1/2} \bigg) \, dt\\
=&\frac{1}{2\pi}\intl_{\R^2} F_+(\vec x) e^{-(x_1^2+x_2^2)/2} d \vec x
+ O(r^2/\bOmega).
\end{split}
\ee
Now, by property $i)$ of Lemma \ref{magic functions},
\bes
\begin{split}
\int_0^T \! F_-\bigg( \Re \zpz(\s+it)V^{-1/2}, & \Im \zpz(\s+it)V^{-1/2} \bigg)  \, dt \\
\leq&  
 \int_0^T \! \mathbf{1}_{{D}_{r}(0)}\( \zpz(\s+it)V^{-1/2} \) \, dt \\
  \leq & \int_0^T \! F_+\( \Re \zpz(\s+it)V^{-1/2}, \Im \zpz(\s+it)V^{-1/2} \) \, dt.
\end{split}
\ees
By this, \eqref{fplus id}, the analogue of \eqref{fplus id} for $F_{-}(u)$, and property $i)$ 
of Lemma \ref{magic functions}, we have that
\begin{equation}\notag
\begin{split}
\int_0^T \!  \mathbf{1}_{{D}_{r}(0)}\(\zpz(\s+it)\) \, dt 
=& \frac{1}{2\pi}\intl_{{D}_{r}(0)}  e^{-(x_1^2+x_2^2)/2} d\vec x \\  
 &+O\bigg(r^2/\bOmega+   \intl_{\R^2} 
\(F_+(\vec x) -  F_-(\vec x)\) e^{-(x_1^2+x_2^2)/2} \, d\vec x\bigg).
\end{split}
\end{equation}
By property $ii)$ of Lemma \ref{magic functions} the integral is $\ll r/\tOmega\ll r /\bOmega$.
The first assertion of the theorem now follows upon noting that
\[
\frac{1}{2\pi}\intl_{{D}_{r}(0)}  e^{-(x_1^2+x_2^2)/2} d\vec x=1-e^{-r^2/2}.
\]

As for the second assertion, by Fourier inversion
\bes
\begin{split}
\frac1T  \, \, \meas\bigg\{t \in (0, T) : 
\AB\frac{\z'}{\z}(\s+it) & \AB\leq \sqrt{V}  r  \bigg\} \\
 \leq &
\intl_{\R^2} \widehat F_{+}(\vec u) \, \, \frac1T \int_0^T e\(\vec u \cdot  \zpz(\s+it)V^{-1/2}\) \, dt \, d\vec u.
\end{split}
\ees
By property $iii)$ of Lemma \ref{magic functions} we may remove the portion
of the integral with $u^2+v^2 > \tOmega^2 $. We then apply Lemma \ref{tail end}
and property $iv)$ of Lemma \ref{magic functions} to see that the right-hand
side of the above inequality is
\bes
\ll\intl_{D_{\tOmega}(0)}  r^2 \( e^{-c(u^2+v^2)}+ \psi(T)^{-10} \)  \, d\vec u \ll r^2.
\ees
Noting that $r \geq 1/\tOmega$, we obtain the result.

\subsection*{Acknowledgments}
This article is part of the author's PhD thesis, which was supervised by Prof. Steven Gonek. I would like 
to thank Prof. Gonek for his guidance and support.

\end{document}